\documentclass[a4paper,12pt]{article}
\usepackage{latexsym}	
\usepackage{amsmath}
\usepackage{amsthm}		

\usepackage{graphicx}

\usepackage{txfonts}
\usepackage{bm}
\usepackage{color}

\usepackage{comment}

\usepackage{geometry}
\numberwithin{equation}{section}

\newtheorem{theorem}{Theorem}[section]
\newtheorem{lemma}[theorem]{Lemma}
\newtheorem{prop}[theorem]{Proposition}

\newtheorem{remark}[theorem]{Remark}
\newtheorem{example}{Example}[section]

\newcommand{\R}{{\mathbb{R}}}

\newcommand{\Z}{{\mathbb{Z}}}

\newcommand{\0}{\bm{0}} 
\newcommand{\1}{\bm{1}}  
\newcommand{\calF}{{\mathcal{F}}}

\newcommand{\Lnat}{{L$^{\natural}$}}
\newcommand{\Mnat}{{M$^{\natural}$}}
\newcommand{\Rinf}{\R \cup \{+ \infty\}}

\newcommand{\pminmin}{p^*_{\rm min}}
\newcommand{\pminminL}{p^{\small \dagger}_{\rm min}} 
\newcommand{\pminL}{p^{\small \dagger}} 

\newcommand{\dom}{{\rm dom\,}}

\newcommand{\suppp}{{\rm supp}\sp{+}}
\newcommand{\suppm}{{\rm supp}\sp{-}}

\newcommand{\Lm}{L}

\newcommand{\OMIT}[1]{{\bf [OMIT:} #1 \ {\bf --- end OMIT] }}  
   \renewcommand{\OMIT}[1]{}            

\begin{document}

\title{
Extension of Excess Demand Ascending Auction \\
to Multi-Demand Model \\
by Discrete Convex Analysis Approach
}

\author{
Kazuo Murota%
\thanks{
Faculty of Economics and Business Administration,
Tokyo Metropolitan University, 
Tokyo 192-0397, Japan,
murota@tmu.ac.jp;
The Institute of Statistical Mathematics, Tokyo 190-8562, Japan.
}
\ and Akiyoshi Shioura%
\thanks{
Department of Industrial Engineering and Economics, 
Institute of Science Tokyo,
Tokyo 152-8550, Japan,
shioura@iee.eng.isct.ac.jp
}%
}


\date{April 2026 / July 2026 / September 2026}


\maketitle

\begin{abstract}
We consider the problem of finding the (unique) minimal Walrasian equilibrium price 
in multi-item, multi-unit auction models:
there are multiple indivisible items for sale, 
with several units of each item, and a bidder may be interested in buying more than one copy of each item.
In its special case with unit-demand bidders,
where each bidder demands at most one unit of any item,
Andersson, Andersson, and Talman (2013)  proposed 
a general framework of ascending auction algorithms
based on the concept of excess-demand item set.
This paper extends this approach to the multi-unit case 
by exploiting the discrete convexity
of the Lyapunov function associated with the auction model.
In particular, we make use of the facts that 
(i) the equilibrium price vectors are characterized as the minimizers of the Lyapunov function,
(ii) the Lyapunov function is an instance of an \Lnat-convex function,
and (iii)
a concept generalizing ``excess-demand item set'' can be defined 
in \Lnat-convex function minimization in general.
\end{abstract}

{\bf Keywords}:
ascending auction,  
Walrasian equilibrium, 
indivisible goods,
Lyapunov function, 
discrete optimization,
discrete convex analysis


\section{Introduction}
\label{sec:intro}


We consider the problem of finding a Walrasian equilibrium price 
in multi-item, multi-unit auction models:
there are multiple indivisible items (goods) for sale, 
with several units of each item, and a bidder may be interested in buying more than one copy of each item.
We focus on ascending auction algorithms, in which the price vector is increased monotonically
to the (unique) minimal Walrasian equilibrium price.

An extensive study has been made on the special case with unit-demand bidders,
where each bidder demands at most one unit of any item.
It is known that there exists a unique minimum equilibrium price vector
and it coincides with the VCG price, which is strategyproof.
This significant property motivates us to compute
the minimum equilibrium price.

Demange, Gale, and  Sotomayor \cite{DGS86} 
gave the first ascending auction algorithm
for computing the minimum equilibrium price in the unit-demand model. 
This algorithm repeatedly
increases the price vector in the components
corresponding to a minimal overdemanded item set,
where a set of items is called overdemanded 
if the number of bidders demanding only items in the set exceeds the number of items in the set;
see \eqref{eqn:overdemand-unit-def} in Section~\ref{sec:review-auction}.
Although this algorithm is guaranteed to compute the minimum equilibrium price,
it has drawbacks that a minimal overdemanded set itself is not easy to compute
and the number of iterations tends to be large.

These drawbacks are resolved in the algorithm of 
Mo, Tsai, and  Lin \cite{MTL88}
(see also Sankaran \cite{San94} and 
Andersson, Andersson, and Talman \cite{AAT13}).  
Instead of a minimal overdemanded set,
this algorithm employs 
a particular overdemanded set
that can be computed with a variant of Ford and Fulkerson's labeling algorithm.
This item set is known to be the largest member 
of the so-called ``purely overdemanded sets''
(\cite[Props.~1 and 2]{MTL88}).
Furthermore, it is shown
by Andersson and Erlanson \cite{AE13} and 
Murota, Shioura, and Yang \cite{MSY13,MSY16},
independently, 
that the number of iterations is the smallest possible.

Andersson et al.~\cite{AAT13}
introduced the concept of excess-demand sets
to present a general framework 
of ascending algorithms
for computing the minimum equilibrium price;
see \eqref{eqn:def-excess-demand} for the definition of an excess-demand set.
The algorithms of 
Demange et al.~\cite{DGS86}
and Mo et al.~\cite{MTL88}
are both captured in this framework,
because a minimal overdemanded set is an excess-demand set
(see Prop.~\ref{prop:excess-demand-overdemanded}(2))
and a purely overdemanded set is 
a synonym of an excess-demand set (\cite[p.~29]{AAT13}).
The framework of Andersson et al.~\cite{AAT13} adds more flexibility
in the choice of an item set,
possibly at the cost of increasing the number of iterations.
The item sets used in the ascending auction algorithm for the unit-demand model
are summarized in Table~\ref{table:unit-demand-itemset}.

\begin{table} 
\begin{center}
\caption{Item sets used in the auction algorithm for the unit-demand model}
\label{table:unit-demand-itemset}
\medskip

\begin{tabular}{ll}
 Authors &  Item set used in the auction algorithm 
\\ \hline
Demange et al.  \cite{DGS86} & (any) minimal overdemanded set 
\\
Mo et al.  \cite{MTL88} & 
(unique) maximal purely overdemanded set 
\\
&   = (unique) maximal excess-demand set
\\
Andersson et al.  \cite{AAT13} & (any) excess-demand set
\\ \hline
 \end{tabular}
\end{center}
\end{table}

The objective of this paper is to show that the flexible ``any excess-demand set'' framework 
also applies to the multi-demand setting.
Specifically, we extend the auction algorithm framework of Andersson et al.~\cite{AAT13} 
to the multi-demand model with strong gross substitutes valuations
by appropriately defining excess-demand sets. 
We derive this extension by exploiting  
the \Lnat-convexity of the Lyapunov function associated with the auction model,
where \Lnat-convexity is one of the central discrete convexity concepts in discrete convex analysis 
\cite{Murota03book,Murota16,Murota24}
(described in Section~\ref{sec:Lnat-minimize}).
We emphasize that this paper is concerned with a structural generalization of the
auction algorithm rather than the development of a new, faster algorithm.

The multi-demand model allows bidders to obtain 
multiple items with several units of each item.
To describe the model more precisely,
denote by $N= \{1, 2, \ldots, n\}$ 
the set of item types and by 
$u(i)$
the number of supplied units of item type $i \in N$. 
Then the items obtained by a bidder $b$
is represented by a vector $x_{b} \in \Z^n$ of nonnegative integers,
and an allocation of the auction is represented by a set of such vectors $\{ x_{b} \}$
satisfying $\sum_{b} x_{b} = u$,
where
$u$ denotes $(u(1), u(2), \ldots, u(n))$,
which is a vector of positive integers.
A special case of 
the multi-demand model
in which
$u = \1$ $(=(1,1,\ldots,1))$
and each $x_{b}$ is a unit vector
($x_{b}\in \{0,1\}^n$, $\| x_{b} \|_{1} = 1$)
is nothing but the unit-demand model.
Another model,
lying between the unit-demand and multi-demand models,
is the unit-supply model,
in which
$u$ is a vector of positive integers and
$x_{b}\in \{0,1\}^n$
for each $b$.
The unit-supply model is studied more extensively in the literature
(e.g., Kelso and Crawford \cite{KC82},
Gul and Stacchetti  \cite{GS99})
than the (general) multi-demand model.

While an equilibrium always exists in the unit-demand model,
this is not the case with 
the multi-demand model or the unit-supply model.
It is known 
(see Section~\ref{sec:model-multi} for details)
that an equilibrium exists 
in the multi-demand model
if bidders' valuation functions satisfy the strong gross substitutes condition,
and that this condition is equivalent to \Mnat-concavity,
which is 
another discrete convexity concept
 in discrete convex analysis.
Under this condition,
there exists a unique minimal member 
among equilibrium price vectors.
The minimum equilibrium price does not generally coincide with the VCG price
in the multi-demand model
(see, e.g., Gul and Stacchetti  \cite{GS00}), 
and accordingly, truthfulness is not guaranteed 
(Ausubel, Cramton, Pycia, Rostek, and Weretka \cite{ACPRW14}). 
Nevertheless, the minimum equilibrium price has several desirable properties; 
in particular, it is the best Pareto-optimal price for bidders.
We always assume 
the strong gross substitutes condition
when we consider the multi-demand or unit-supply models.

In this paper, we are interested in 
extending the 
algorithms in Table~\ref{table:unit-demand-itemset}
to those for finding the minimal equilibrium price
in the multi-demand model.
Gul and Stacchetti \cite{GS00} 
gave an ascending auction algorithm for the unit-supply model 
and Ausubel~\cite{Aus06} for the (general) multi-demand model.
When used in the unit-demand model,
both of these algorithms
produce the same trajectory (ascending sequence) of the price vector
as the algorithm of Mo et al.~\cite{MTL88}.
However, no ascending auction algorithms have been given 
for the multi-demand model
that correspond to the algorithm 
of Demange et al.~\cite{DGS86} or that of 
Andersson et al.~\cite{AAT13}
for the unit-demand model.

Our approach is based on the following fundamental connection between auctions and \Lnat-convex functions:
\begin{itemize}
\item
An equilibrium price is characterized as a minimizer of
the Lyapunov function associated with the model
(Ausubel \cite{Aus06}).

\item
The Lyapunov function is an \Lnat-convex function
(Murota--Shioura--Yang \cite{MSY16}).
\end{itemize}
We thus follow the approach expounded in Murota
\cite[Sec.~8]{Murota16}, \cite[Sec.~17.3]{Murota24} and
Bichler,  Fichtl, and Schwarz \cite{BFS21}.
We emphasize that the above facts hold true
both in the unit-demand model
and in the multi-demand model with the strong gross substitutes property.

Instead of directly extending the excess-demand auction algorithm, 
we make a detour to the Lyapunov function minimization
to obtain the excess-demand auction algorithm for the multi-item model;
see Table~\ref{table:secorg}.
We carry out this idea in the following three steps.

\begin{table} 
\begin{center}
\caption{Derivation of the auction algorithm for the multi-demand model}
\label{table:secorg}

\medskip

\addtolength{\tabcolsep}{-2pt}
 \begin{tabular}{cccc}
 Model & Auction algorithm &  & Function minimization
\\ 
 & using excess-demand set  & & (Lyapunov / \Lnat-convex)
\\
\hline
Unit-demand model & Sec.~\ref{sec:review-auction} (review) 
& $\longrightarrow$ & Sec.~\ref{sec:Lyap-unit} (Lyapunov)
\\
   &   &   &  $\downarrow$
\\
   &   &   &  Sec.~\ref{sec:Lnat-minimize} (\Lnat-convex)
\\
   &   &   &  $\downarrow$
\\
Multi-demand model & Sec.~\ref{sec:auction-multi} (main results) 
& $\longleftarrow$ & Sec.~\ref{sec:Lyap-multi} (Lyapunov)
\\ \hline
 \end{tabular}
\addtolength{\tabcolsep}{+2pt}
\end{center}
\end{table}

\paragraph{1st step:}
By reformulating the concepts of excess-demand set 
in terms of the Lyapunov function,
we rewrite the ascending auction algorithm
for the unit-demand model 
to a Lyapunov function minimization algorithm.
For each variant 
of the ascending auction algorithm
in Table \ref{table:unit-demand-itemset},
we obtain a corresponding variant of
the Lyapunov function minimization algorithm.

\paragraph{2nd step:}
We show that the Lyapunov function minimization algorithms 
for the unit-demand model
(obtained in the 1st step)
can be used,
without substantial changes, 
as algorithms to find the minimal minimizer of a general \Lnat-convex function.

\paragraph{3rd step:}
By applying the \Lnat-convex function minimization algorithm
(constructed in the 2nd step)
to the Lyapunov function of the multi-demand model,
we obtain the Lyapunov function minimization algorithm
for finding the minimal equilibrium price
in the multi-demand model.
Then we turn the Lyapunov function minimization algorithm
into an excess-demand auction algorithm
by appropriately defining the concepts of deficiency, overdemanded sets, and excess-demand sets
in the multi-demand setting in terms of the Lyapunov function.

\medskip

The organization of this paper is as follows.
Section~\ref{sec:review-auction} 
explains the unit-demand auction model, 
followed by a review on the existing ascending auction algorithms 
using excess-demand item sets.
Section~\ref{sec:auction-multi} describes the main results
of this paper, showing that,  
by appropriately defining ``excess-demand set'' in the multi-demand setting,
we can extend the existing algorithms for the unit-demand model
to those for the multi-demand model.
The subsequent three sections,
Sections \ref{sec:Lyap-unit} to \ref{sec:Lyap-multi}, 
are devoted to the justification of the main results
with the aid of discrete convex analysis.
In Section~\ref{sec:Lyap-unit}, 
the ascending auction algorithm using excess-demand sets 
for the unit-demand model is 
reformulated to a minimization algorithm of the Lyapunov function
in the unit-demand model.
This minimization algorithm is extended,  
in Section~\ref{sec:Lnat-minimize}, 
to general \Lnat-convex functions.
Finally, in Section~\ref{sec:Lyap-multi},
this general algorithm is tailored 
to the Lyapunov function in the multi-demand model
to obtain excess-demand ascending auction algorithms for the multi-demand model.
The structure of this paper is schematized in Table~\ref{table:secorg}.

\subsection*{Related work}
Several recent papers on auctions make effective use of concepts and techniques 
in discrete and continuous optimization
to improve computational efficiency.
Paes Leme and Wong \cite{PLW20wal} 
present a polynomial algorithm for equilibrium computation
in an aggregate demand oracle model
by making judicious use of subgradients
as well as another polynomial algorithm in the value oracle model.
Eickhoff, McCormick, Peis, Rieken, and Vargas Koch \cite{EMPRK24}
give a network flow-based algorithm 
for a specific type of gross substitutes valuation,
called a truncated additive valuation.
Eickhoff, Neuwohner, Peis, Rieken, Vargas Koch, and V{\'e}gh \cite{ENPRKV25} 
present a primal-dual framework that features 
the polymatroid sum problem, 
in place of \Lnat-function minimization,
in identifying the minimal maximally overdemanded set
(= the minimal maximizer of deficiency $\delta$ in \eqref{eqn:def-OYp}
in our paper). 
Baldwin, Goldberg, Klemperer, and Lock \cite{BGKL24}
present a polynomial algorithm for computing equilibrium allocations
in product-mix auctions,
in which the valuations of the bidders are expressed as bid vectors.
Fujishige and Yang \cite{FY26}
propose dynamic auctions 
applicable to valuations of all unimodular demand types,
where gross substitutes valuations are a special case of such valuations.

\section{Review of ascending auction algorithms for unit-demand model}
\label{sec:review-auction}


 In this section, we review the unit-demand auction model and
the ascending auction algorithms 
for computing the unique minimal Walrasian price vector in that model.

\subsection{Unit-demand auction model and Walrasian equilibrium}
\label{sec:unit-equil}

 In the unit-demand auction model,
each bidder demands at most one item.
 We assume that there are 
$m$ bidders and $n$ items, and denote the sets of bidders and items as
$M=\{1, 2,\ldots, m\}$ and $N= \{1, 2, \ldots, n\}$, respectively.
 Each item $i \in N$ is available only in one unit,
and its  value for bidder $b \in M$ 
is denoted as
$v_b(i) \in \Z_+$.
 In addition, we consider an artificial item $0$
that has no value (i.e., $v_b(0)=0$ for 
all $b \in M$)
and is available in  infinite number of units.
We denote $N^0=N \cup \{0\}$.

 Given a price vector $p \in \Z^n_+$, 
each bidder $b \in M$ demands
an item $i \in N^0$
that maximizes the value $v_b(i) - p(i)$, where $p(0) = 0$ by convention.
 For each bidder $b$ and a price vector $p \in \Z^n_+$, 
define the demand set  $\widetilde{D}_b(p)$ as
\begin{align*}
  \widetilde{D}_b(p)
& = \arg\max\{v_b(i) - p(i)  \mid i \in N^0  \}\\
& = \{i \in N^0 \mid  v_b(i) - p(i) \geq v_b(j) - p(j)
 \ (\forall j \in N^0 )  \} .
\end{align*}
Note that $\widetilde{D}_b(p)$ is a nonempty subset of $N^0$ for any $p$.

 An \textit{allocation} of items to bidders is defined as a mapping
$\pi: M \to N^0$ such that $\pi(b) \ne \pi(b')$ 
for every distinct $b, b' \in M$
with $\pi(b) \in N$ and $\pi(b') \in N$;
this condition means that each item is allocated to at most one bidder.
 The auctioneer wants to find 
a pair of a price vector $p^* \in \Z_{+}^n$ and an allocation $\pi$
such that $\pi(b) \in \widetilde{D}_b(p^*)$ for every $b \in M$
and $p^*(i)=0$ for every item $i \in N \setminus \{ \pi(b) \mid b \in M \}$.
Such a pair 
is called a \textit{(Walrasian) equilibrium} and 
$p^*$ is a \textit{(Walrasian) equilibrium price vector}.
 An equilibrium price vector
always exists, and a minimal equilibrium price vector
is uniquely determined (see, e.g., \cite{DG1985,SS1972}),
which is denoted as $\pminmin$ in the following.

\subsection{Ascending auction algorithms}
\label{sec:unit-alg}

 We start with a generic form of an ascending auction algorithm 
(Algorithm 1 in \cite{AAT13}).
 Let $p \in \Z^n_+$ be a price vector and $Y \subseteq N$ an item  set.  
 We define
\begin{equation}
\label{eqn:def-OYp}
  O(Y, p)  = \{b \in M \mid \widetilde{D}_b(p) \subseteq Y\},
\end{equation}
which denotes the set of bidders who only demand items in $Y$ at price $p$.
The
\textit{deficiency} of $Y$  
is defined as the value of
\begin{equation}
\label{eqn:deficiency-unit-def}
\delta(Y;p)=
|O(Y, p)|-|Y|.
\end{equation}
 If the deficiency of $Y$ is positive, that is,
\begin{equation}
\label{eqn:overdemand-unit-def}
 |O(Y, p)| > |Y|,
\end{equation}
at least one bidder in $O(Y,p)$ cannot get any item in~$Y$;
such an item set $Y$ is said to be
\textit{overdemanded} at $p$.
Note that an overdemanded set is necessarily a nonempty set,
since $\widetilde{D}_b(p) \ne \emptyset$ for any $p$ and $b \in M$.

The following is a basic ascending auction algorithm,
which aims to eliminate overdemanded sets
by iteratively increasing the price vector $p$.
For any $X \subseteq N$,
we write $\chi_X \in \{0,1\}^n$ for the characteristic vector of $X$.

\medskip

\noindent
{\bf Algorithm} {\sc Ascending Auction} (generic form)
\\
Step~0:  Set $p:=p^\circ$, where $p^\circ \in \Z^n_+$ 
is a vector satisfying $p^\circ \le \pminmin$. 
\\
Step~1:
If there is no overdemanded set at price $p$,  then stop.
\\ 
Step~2:
 Find an item set $X \subseteq N$ that is overdemanded at price $p$.
\\
Step~3:
 Set $p:=p + \chi_X$ and go to Step~1.

\medskip

This algorithm, {\sc Ascending Auction}, terminates after a finite number of iterations
since all overdemanded sets disappear when the prices become sufficiently large.
However, the vector $p$ obtained at the termination of the algorithm
may not be an equilibrium price,
since the nonexistence of overdemanded sets does not imply equilibrium.

With a judicious choice of the item set $X$ in Step~2,
the minimal equilibrium price vector  $\pminmin$
can be obtained by the above algorithm.
We explain this in three steps
according to the historical development:
(i) Demange et al.~\cite{DGS86} (1986),
(ii) Mo et al.~\cite{MTL88} (1988),
and (iii) Andersson, Andersson, and Talman \cite{AAT13} (2013).

Demange et al.~\cite{DGS86} 
proposed selecting a \textit{minimal} overdemanded set for the item set $X$ in Step~2
and showed that this choice yields
the minimal equilibrium price $\pminmin$.
It is noted that a minimal overdemanded set is not uniquely determined;
see Example \ref{ex: excess-demand}
at the end of this section.

\begin{theorem}[\cite{DGS86}]
\label{thm:min-overdemand-unit}
 The algorithm {\sc Ascending Auction} finds 
the minimal equilibrium price vector  $\pminmin$
if, in each iteration, the item set $X$ 
selected in Step~2
is a minimal overdemanded set at price~$p$.
\end{theorem}

 Use of excess-demand sets in {\sc Ascending Auction} was initiated by 
Mo et al.~\cite{MTL88} (see also Sankaran \cite{San94} and Andersson et al.~\cite{AAT13}).
 For a price vector $p \in \Z^n_+$ and an item  set $Y \subseteq N$,  
 define
\begin{equation}
\label{eqn:def-UYp}
 U(Y, p)  = \{b \in M \mid \widetilde{D}_b(p) \cap Y \neq \emptyset\},
\end{equation}
which denotes the set of bidders who demand some item in $Y$ at price~$p$.
We have $U(Y, p) \supseteq O(Y,p)$ from \eqref{eqn:def-OYp} and \eqref{eqn:def-UYp}.

A nonempty set $X \subseteq N$ is said to be an \textit{excess-demand set}%
\footnote{
An excess-demand set is called a \textit{purely overdemanded set} in Mo et al.~\cite{MTL88}
and a \textit{strictly overdemanded set} in D{\"u}tting et al.~\cite{DHW16}.
Alkan \cite{Alk92} uses the term  ``overdemanded set'' 
to mean an excess-demand set (not our overdemanded set),
and Gul and Stacchetti \cite{GS00} use the term ``excess-demand set'' to mean
the unique maximal excess-demand set (not our excess-demand set).
}
at price $p$ if it satisfies
\begin{equation}
\label{eqn:def-excess-demand}
  |U(Z, p) \cap O(X, p)| > |Z| \qquad (\emptyset \neq \forall Z \subseteq X).
\end{equation}
 Here, $|U(Z, p) \cap O(X, p)|$ is the maximum number of items the bidders
are willing to take from the set $Z\ (\subseteq X)$,
under the assumption that only the bidders in $O(X, p)$
get items in $X$.
 Hence, the condition \eqref{eqn:def-excess-demand}
 means that
for every nonempty subset $Z$ of $X$,
the number of items in $Z$ demanded by the bidders
is strictly larger than the size of $Z$.
An alternative expression of the condition \eqref{eqn:def-excess-demand}, 
without involving $U(\cdot, p)$, is given by
\begin{equation}
\label{eqn:def-excess-demand2}
|O(X, p)| - |O(X \setminus Z, p)|
 > |Z| \qquad (\emptyset \neq \forall Z \subseteq X) ,
\end{equation}
which is indeed equivalent to \eqref{eqn:def-excess-demand} since
\begin{align}
& U(Z, p) \cap O(X, p) =  O(X, p)\setminus O(X \setminus Z, p),
\label{eqn:def-excess-demand3a}
\\ 
& O(X,p)  \supseteq O(X \setminus Z, p)
\label{eqn:def-excess-demand3b}
\end{align}
for every nonempty subset $Z$ of $X$.
The alternative expression \eqref{eqn:def-excess-demand2}
will be used in Section~\ref{sec:auction-alg--multi} 
to extend the notion of excess-demand set to the multi-demand model.

Various properties of excess-demand sets,
including the relationship with overdemanded sets,
are listed in the following proposition.
It is noteworthy that
there is a unique maximal excess-demand set,
while there may be several different minimal excess-demand sets;
see Example \ref{ex: excess-demand}
at the end of this section.

\begin{prop}[cf.~\cite{AAT13,vdLY16,MTL88}]
\label{prop:excess-demand-overdemanded}
\ \\
{\rm (1)}
An excess-demand set is an overdemanded set.
\\
{\rm (2)}
A minimal overdemanded set is an excess-demand set.
\\
{\rm (3)}
An item set is a minimal excess-demand set if and only if
it is a minimal overdemanded set.
(This is a consequence of {\rm (1)} and {\rm (2)}  above.)
\\
{\rm (4)}
For every excess-demand sets $X$ and  $Y$, 
their union $X \cup Y$ is also an excess-demand set.
Consequently, the maximal excess-demand set is uniquely determined.
\\
{\rm (5)}
An item set $X$ is the unique maximal  excess-demand set 
if and only if 
it is the unique minimal maximizer of the deficiency
$\delta(X;p)=|O(X,p)| - |X|$.
\end{prop}

Mo et al.~\cite{MTL88} (see also \cite{San94})
 showed that the minimal  equilibrium price vector $\pminmin$ can be obtained
by using the unique maximal excess-demand set in each iteration of the algorithm.

\begin{theorem}[\cite{MTL88,San94}]
\label{thm:max-excess}
The algorithm {\sc Ascending Auction} finds 
the minimal  equilibrium price vector  $\pminmin$
if, in each iteration, the item set $X$ selected in Step~2
is  the maximal excess-demand set at price $p$.
\end{theorem}

It follows from Theorems \ref{thm:min-overdemand-unit} and \ref{thm:max-excess}
with Proposition~\ref{prop:excess-demand-overdemanded}
that the minimal equilibrium price vector can be obtained
by the algorithm {\sc Ascending Auction} 
if the set $X$ selected in each iteration
is either a minimal excess-demand set or the unique maximal excess-demand set.
 Andersson et al.~\cite{AAT13} showed that 
the minimality/maximality condition is not needed here.

\begin{theorem}[\cite{AAT13}]
\label{thm:excess}
The algorithm {\sc Ascending Auction} finds 
the minimal equilibrium price vector  $\pminmin$
if, in each iteration, the item set $X$ selected in Step~2
is an excess-demand set at price $p$.
\end{theorem}

Andersson et al.~\cite[Example~1]{AAT13} argue that 
the choice of an excess-demand set in Step~2
is essential for finding $\pminmin$.
The behavior of this algorithm is investigated in detail  
for a concrete instance in \cite[Example~2]{AAT13}.

The following example illustrates 
the properties of 
overdemanded sets and excess-demand sets
stated in Proposition~\ref{prop:excess-demand-overdemanded}.

\begin{example} \rm 
\label{ex: excess-demand}
Consider a unit-demand auction with three items $N = \{ 1,2,3 \}$ 
and six bidders $M = \{ a, b, c, d,e,g \}$, and assume that the valuations are given by
\begin{align*}
 v_{a}(i) = v_{b}(i) &=
\begin{cases}
  1 & (i=1),\\
  0 & (i=2,3),\\
\end{cases}
\\
 v_{c}(i) = v_{d}(i) = v_{e}(i) &=
\begin{cases}
  0 & (i=1),\\
  1 & (i=2,3),\\
\end{cases}
\\
 v_{g}(i)  &=
\begin{cases}
  1 & (i=1,2),\\
  0 & (i=3).\\
\end{cases}
\end{align*}
Let $p = \0$, for which
\[
 \widetilde{D}_{a}(p)=\widetilde{D}_{b}(p)
= \{1\},
\quad
 \widetilde{D}_{c}(p)=\widetilde{D}_{d}(p)=\widetilde{D}_{e}(p)
= \{2,3\},
\quad
 \widetilde{D}_{g}(p)= \{1,2\}.
\]
The sets $O(Y,p)$ and $U(Y,p)$, 
defined in \eqref{eqn:def-OYp} and \eqref{eqn:def-UYp} respectively,
are given as
\begin{align*}
 O(Y,p) &=
\begin{cases}
  \emptyset & (Y = \emptyset, \{2\}, \{3\}),\\
  \{a,b\} & (Y = \{1\},  \{1,3\}),\\
  \{a,b,g \} & (Y = \{1,2\}),\\
  \{c,d,e\} & (Y =  \{2,3\}),\\
 \{a,b,c,d,e,g \} & (Y =  \{1,2,3 \}),
\end{cases}
\\
 U(Y,p) &=
\begin{cases}
  \emptyset & (Y = \emptyset),\\
 \{a,b,g \} & (Y = \{1\}),\\
  \{c,d,e,g\} & (Y = \{2\}, \{2,3\}),\\
  \{c,d,e\} & (Y = \{3\}),\\
 \{a,b,c,d,e,g \} & (Y = \{1,2\}, \{1,3\}, \{1,2,3 \}).
\end{cases}
\end{align*}
By straightforward calculation, we see that 
the family of overdemanded sets is given by 
$\{  \{ 1 \},   \{ 1,2 \}, \allowbreak   \{ 2, 3 \}, \{ 1, 2, 3 \}   \}$
and that of excess-demand sets by
$\{  \{ 1 \},   \{ 2, 3 \}, \{ 1, 2, 3 \}   \}$.
The former is a proper superset of the latter 
(Proposition~\ref{prop:excess-demand-overdemanded} (1))
and the minimal elements of the former,
$\{ 1 \}$ and $\{ 2, 3 \}$,
are members of the latter
(Proposition~\ref{prop:excess-demand-overdemanded} (2)).
There are two minimal excess-demand sets,
$\{ 1 \}$ and $\{ 2, 3 \}$,
and a unique maximal excess-demand set
$\{ 1, 2, 3 \}$.
\qed
\end{example}

\section{Ascending auction algorithms for multi-demand model}
\label{sec:auction-multi}


This section conveys the main message of this paper that
we can make the algorithm {\sc Ascending Auction} 
applicable to the multi-demand model
by appropriately defining ``overdemanded set'' and ``excess-demand set'' 
in the multi-demand setting.

\subsection{Multi-demand auction model}
\label{sec:model-multi}

 In contrast to the unit-demand model where each bidder is interested in obtaining a single item,  
the multi-demand model allows bidders to obtain 
multiple items with several units of each item.
In this setting, $N= \{1, 2, \ldots, n\}$ denotes the set of types of items,
rather than the set of individual items,
and the number of supplied units of item $i \in N$ is denoted by 
a positive integer $u(i)$. 
We denote 
\begin{align*}
& u  =(u(1), u(2), \ldots, u(n)),
\\
& [\0, u]_\Z  = \{x \in \Z^n \mid \ 0 \leq x(i) \leq u(i)\ (i\in N)\}.  
\end{align*}
 Each vector $x \in [\0, u]_\Z$ is called a \textit{bundle};
a bundle $x=(x(1), x(2), \ldots, x(n))$ 
corresponds to a multi-set 
of items, and $x(i) \in \Z_+$ 
represents the multiplicity of item $i \in N$.

 Each bidder $b \in M$ has a valuation function
$v_b: [\0, u]_\Z \to \Z$;
the number $v_b(x)$ represents the
value of the bundle $x$ worth to bidder $b$.
We assume that  each $v_b$ is monotone nondecreasing with 
$v_b(\0)=0$.
 For $b \in M$ and $p \in \Z^n_+$, 
define  the \textit{demand set} as
\begin{align*}
D_b(p) & =  \arg\max\{v_b(x) - p^\top x \mid x \in [\0, u]_\Z\}.
\end{align*}

 An \textit{allocation} of items is defined as a set of bundles
$x_1, x_2, \ldots, x_m \in [\0, u]_\Z$
satisfying $\sum_{b=1}^m x_b = u$.
 The auctioneer wants to find 
a pair of a price vector $p^*$ and an allocation
$x_1^*, x_2^*, \ldots, x_m^*$ such that 
$x_b^* \in D_b(p^*)$ for all $b \in M$.
Such a pair is called a \textit{(Walrasian) equilibrium}
and $p^*$ is a \textit{(Walrasian) equilibrium price vector}.

\begin{remark}\rm
Two different definitions of Walrasian equilibrium can be found in the literature. 
 The definition in this paper follows the one in \cite{Aus06}.
An alternative definition, in which unsold items may exist, is
as follows (cf.~\cite[Section 3]{GS99}):
a pair of a price vector $p^*$ and a set of bundles
$\tilde x_1, \ldots, \tilde x_m^*$ is called a {\it (Walrasian) equilibrium}
if it satisfies the following conditions:
\begin{equation}
\label{eqn:We2}
 \left.
\begin{array}{ll}
 \sum_{b=1}^m \tilde x_b \leq u,\\
 \tilde x_b \in D_b(p^*) \mbox{ for all }b \in M,\\
 p^*(i)=0 \mbox{ for all }i \in N \mbox{ with }\sum_{b=1}^m \tilde x_b(i) < u(i).
\end{array} 
\right\}
\end{equation}
 The two definitions of equilibrium are essentially 
equivalent in the sense that 
the set of equilibrium price vectors remains the same
under the two definitions,
provided that valuation functions 
$v_b$ $(b \in M)$
 are nondecreasing
(see, e.g., \cite[Section 3]{GS99}).
 Indeed, 
it is easy to see that if
a pair of a price vector $p^*$ and an allocation
$x_1^*, x_2^*, \ldots, x_m^*$ is an equilibrium in our sense,
then it satisfies the conditions in (\ref{eqn:We2}).
 Conversely, 
if $p^*$ and $\tilde x_1, \ldots, \tilde x_m^*$ satisfy 
(\ref{eqn:We2}),
the allocation $x_1^*, x_2^*, \ldots, x_m^*$ defined by
\[
 x_b^* = \tilde x_b\ (b=1,2,\ldots, m-1), \quad x_m^* = u - \sum_{b=1}^{m-1} \tilde x_b
\]
satisfies $x_b^* \in D_b(p^*)$  for $b = 1,2,\ldots,  m$; 
note that $x_m^* \in D_m(p^*)$ follows from 
(\ref{eqn:We2}) and
the assumed monotonicity of $v_m$.  
 Hence, $p^*$ is an equilibrium price in our sense.
Our unit-demand model described in Section~\ref{sec:unit-equil}
is of the form of \eqref{eqn:We2} with $u = \1$.
\qed
\end{remark}

\par
 While an equilibrium always exists in the unit-demand model,
it may not in the multi-demand model.
For the existence of an equilibrium in the multi-demand model,
a property of valuation functions,
called \textit{strong gross substitutes condition}, 
plays a critical role. 
To be specific,  it is known 
(Milgrom and Strulovici \cite[Theorem~19]{MiSt09}; also Ausubel \cite{Aus06})
that if bidders' valuation functions satisfy the strong gross substitutes condition,
an equilibrium exists 
and moreover, there exists a unique minimal equilibrium price.
The condition  of strong gross substitutes
was introduced by Milgrom and Strulovici \cite{MiSt09} 
as a generalization of the gross substitutes condition by Kelso and Crawford \cite{KC82} 
(also Gul and Stacchetti \cite{GS99})
for valuation functions defined on $\{0,1\}^n$.

The strong gross substitutes condition is known (\cite[Section~4]{ST15}) to be equivalent to 
\Mnat-concavity,
which is a fundamental concept in discrete convex analysis.
A valuation function $v: [\0, u]_\Z \to \Z$
 is said to be
\textit{\Mnat-concave} \cite{Murota03book,Murota24} if it satisfies the following 
exchange axiom:
\begin{align*}
&  \mbox{\bf (\Mnat-EXC)}
\ \forall x, y \in 
[\0, u]_\Z,
\ \forall i \in \suppp(x-y),\ 
\exists k \in \suppm(x-y)\cup\{0\}: \\
& \hspace*{20mm}
 v(x) + v(y) 
 \leq v(x - \chi_i + \chi_k) + v(y + \chi_i - \chi_k),
\end{align*}
where for 
any vector $z \in \Z^n$, we define
\[
 \suppp(z) = \{i \in N \mid z(i)>0\}, \qquad
\suppm(z) = \{i \in N \mid z(i)<0\},
\]
and for any $i \in N$, the notation 
$\chi_{i}$ means the $i$th unit vector,
that is, $\chi_{i} = \chi_{\{ i \} }$,
and $\chi_0 = {\bm 0}$ by convention.
Given the equivalence of \Mnat-concavity to strong gross substitutability,
we can say that an equilibrium exists
if the valuation functions $v_b$ $(b \in M$) 
are \Mnat-concave.
In this paper we always assume that valuation functions
$v_b$ $(b \in M)$ are \Mnat-concave.
It is noted that the valuation function 
underlying the unit-demand model is \Mnat-concave
(see \cite[Note~11.9]{Murota03book}, \cite[Example~2.3]{ST15}).

 \Mnat-concave/convex functions constitute a major class of 
discrete concave/convex functions in the theory of discrete convex analysis,
and they find various applications in
combinatorial optimization, operations research, algebra, etc., in addition to
mathematical economics and auction theory;
see \cite{Murota03book,Murota16,Murota24}.
In general, \Mnat-concave functions can be real-valued functions 
defined on non-rectangular domains,
while in this paper it suffices to consider
integer-valued functions defined on rectangular domains.

\subsection{Ascending auction algorithm}
\label{sec:auction-alg--multi}

In this section, we define the notions of ``overdemanded set'' and ``excess-demand set'' 
for the multi-demand model.
This enables us to extend the algorithm 
 {\sc Ascending Auction} of Section~\ref{sec:review-auction}
to the multi-demand setting.

 We first define an overdemanded set in the multi-demand model.
For $p \in \Z^n_+$ and $X\subseteq N$, define a notation
\begin{equation}
\label{eqn:def-mubXp}
 \mu_b(X;p) = \min\{y(X) \mid y \in D_b(p)\}
\qquad (b \in M),
\end{equation}
where $y(X) = \sum_{i \in X} y(i)$.
 The value $\mu_b(X;p)$ denotes the minimum quantity of  item units 
demanded by bidder $b$ from item set $X$ at price $p$.
Their sum $\sum_{b \in M} \mu_b(X;p)$
is the minimum total quantity of item units demanded from item set $X$,
while $u(X)=\sum_{i \in X}u(i)$ 
is the maximum quantity of item units that can be allocated to item set $X$.
 Hence, 
it would be natural to say that a set $X \subseteq N$ is 
\textit{overdemanded} at price $p$ if 
\begin{equation}
\label{eqn:overdemand-multi-def}
\sum_{b \in M} \mu_b(X;p) > u(X).
\end{equation}
We refer to 
\begin{equation}
\label{eqn:deficiency-multi-def}
\delta(X;p)=
\sum_{b \in M} \mu_b(X;p) - u(X)
\end{equation}
as the \textit{deficiency} of $X$.
It can be shown (see Proposition~\ref{prop:mLyap-diff-cor}(2))
that there exists a unique minimal member among the maximizers of 
the deficiency $\delta(X;p)$.
In the unit-demand model,
$\sum_{b \in M} \mu_b(X;p)$ equals $|O(X,p)|$ defined in \eqref{eqn:def-OYp}
and $u(X) = |X|$,
and therefore,
\eqref{eqn:overdemand-multi-def} coincides with the inequality 
$|O(X, p)| > |X|$ in \eqref{eqn:overdemand-unit-def}.

The algorithm {\sc Ascending Auction} in Section~\ref{sec:review-auction}
is described in terms of overdemanded sets.
Although it was originally meant 
for the unit-demand model,
the algorithm makes sense,
at least formally, 
for the multi-demand model
with the above-definition of overdemanded sets
in the multi-demand model.
The following theorem, 
which is the first main result of this paper,
states that the algorithm thus obtained is indeed valid 
for the multi-demand model.

\begin{theorem}
\label{thm:min-overdemand-multi}
 The algorithm {\sc Ascending Auction} 
applied to the multi-demand model 
finds the unique minimal equilibrium price vector  $\pminmin$
if, in each iteration, 
{\rm (i)}
the item set $X$ selected in Step~2 is 
a minimal overdemanded set \eqref{eqn:overdemand-multi-def} at price~$p$, or
{\rm (ii)}
 $X$ is the unique minimal maximizer of 
the deficiency $\delta(X;p)$ in \eqref{eqn:deficiency-multi-def}.
\end{theorem}
\begin{proof}
The proof is given in 
Section~\ref{sec:Lyap-multi}.
In Sections \ref{sec:Lyap-unit} and \ref{sec:Lnat-minimize},
we introduce the necessary tools for the proof,
including the Lyapunov function
and its discrete convexity, called \Lnat-convexity.
\end{proof}

Next, we go on to define ``excess-demand set''
for the multi-demand model.
We say that a nonempty item set $X$ is an \textit{excess-demand set}
if the following condition is satisfied:
\begin{align}
\label{eqn:excess-demand-multi-def}
  \sum_{b \in M} (\mu_b(X;p) - \mu_b(X\setminus Z;p)) > u(Z) 
 \qquad (\emptyset \ne  \forall Z \subseteq X),
\end{align}
where \eqref{eqn:excess-demand-multi-def} is equivalent to
\begin{align}
\label{eqn:excess-demand-multi-def2}
  \delta(X;p) - \delta(X\setminus Z;p) > 0 
 \qquad (\emptyset \ne  \forall Z \subseteq X).
\end{align}

To justify this definition, suppose that each bidder $b$ wants
 to receive 
$x_b(X)$ units from item set $X$
and $x_b(X\setminus Z)$ units from item set $X\setminus Z$
simultaneously,
subject to the condition that $x_b \in D_b(p)$.
  Then, 
$\mu_b(X;p) - \mu_b(X\setminus Z;p)$
is the number of item units demanded by bidder $b$ from the item set $Z$.
Thus the inequality \eqref{eqn:excess-demand-multi-def} 
means that 
the sum of this demand over all bidders
exceeds 
the number of units $u(Z)$ of supply from $Z$, 
which is the situation where the demand exceeds the supply.
In the unit-demand model,
the condition \eqref{eqn:excess-demand-multi-def} 
reduces to the alternative expression  
\eqref{eqn:def-excess-demand2}
of the defining condition for excess-demand sets.

An excess-demand set is a special case of an overdemanded set,
since the inequality 
\eqref{eqn:excess-demand-multi-def} with $Z = X$
coincides with \eqref{eqn:overdemand-multi-def}.
Thus the algorithm {\sc Ascending Auction}
using an excess-demand set in Step~2
is applicable,
at least formally, 
to the multi-demand model.
The following theorem, 
which is the second main result of this paper,
states that the algorithm thus obtained is indeed valid 
for the multi-demand model.

\begin{theorem}
\label{thm:excess-multi}
The algorithm {\sc Ascending Auction} applied to the multi-demand model 
finds the unique minimal equilibrium price vector $\pminmin$
if, in each iteration, the item set $X$ selected in Step~2
is an excess-demand set
in the sense of \eqref{eqn:excess-demand-multi-def}.
\end{theorem}
\begin{proof}
The proof is given in 
Section~\ref{sec:Lyap-multi}
after the necessary tools,
including the Lyapunov function
and its \Lnat-convexity,
are introduced in Sections \ref{sec:Lyap-unit} and \ref{sec:Lnat-minimize}.
\end{proof}

It should be clear that
Theorems \ref{thm:min-overdemand-multi} 
and \ref{thm:excess-multi} above are
the multi-demand extensions of
Theorems \ref{thm:min-overdemand-unit} and \ref{thm:excess}
for the unit-demand model, respectively.

\begin{remark} \rm
\label{rem:runningtimeMult}
The running time of the {\sc Ascending Auction} algorithm
for the multi-demand model 
using an excess-demand set $X$ in Step~2
is discussed here.
The number of iterations (updates) lies between 
$\| p^\circ - \pminmin \|_{\infty}$
and
$\| p^\circ - \pminmin \|_{1}$ $(\le n \cdot \| p^\circ - \pminmin \|_{\infty})$.
If, in each iteration,
$X$ is the maximal excess-demand set, 
the number of iterations equals the lower bound 
$\| p^\circ - \pminmin \|_{\infty}$;
see \cite{MSY16}.
In each iteration, we have to find an excess-demand set.
We can compute the maximal excess-demand set in polynomial time
via submodular function minimization \cite{MSY16}
or via the polymatroid sum problem \cite{ENPRKV25}.
We can also compute a minimal excess-demand set in polynomial time with 
$O(n\sp{2})$ 
applications of submodular function minimization.
It is not known whether we can compute faster
some other excess-demand set (without imposing maximality or minimality).
\qed
\end{remark}

\section{Excess-demand ascending auction viewed as Lyapunov function minimization}
\label{sec:Lyap-unit}


The following three sections
are devoted to the proofs of Theorems \ref{thm:min-overdemand-multi} 
and \ref{thm:excess-multi}.
In this section, 
we present the Lyapunov function approach for the unit-demand auction model.
By reformulating the notions of excess-demand set and its extremal cases
(minimal and maximal excess-demand sets)
in terms of the Lyapunov function,
we rewrite the algorithm {\sc Ascending Auction}
of Section~\ref{sec:review-auction} 
to a minimization algorithm for the Lyapunov function.

The \textit{Lyapunov function} 
for the unit-demand model
is a function in price vector $p \in \Z^n_+$
defined as
\begin{align}
L(p) &=
\sum_{b \in M} \max\big[0, \max_{i \in N}\{v_b(i) - p(i)  \}\big] + p(N)
\qquad (p \in \Z^n_+),
\label{eqn:Lyap-unit}
\end{align}
where $p(N) = \sum_{i \in N} p(i)$.
The set of  minimizers of 
this function
coincides with the set of  equilibrium price vectors,
which is a special case of a well-known result of Ausubel~\cite{Aus06}.

\begin{prop}
In the unit-demand model, 
a vector $p \in \Z^n_+$ is an  equilibrium price vector 
if and only if it is a  minimizer of the Lyapunov function $L$
in \eqref{eqn:Lyap-unit}.
In particular,
$p \in \Z^n_+$ is the minimal equilibrium price vector 
if and only if it is the minimal minimizer of $L$.
\end{prop}

Our next step is to show that
 the algorithm {\sc Ascending Auction} for finding an
equilibrium price vector can be regarded as
an incremental algorithm for finding  
a minimizer of the Lyapunov function,
where the  vector $p$ is repeatedly incremented with a 0-1 vector.
 In the following, we refer to 
a 0-1 vector $d$ with $L(p+ d) < L(p)$ as 
a \textit{descent direction} of the Lyapunov function $L$ at $p$.

The next proposition connects
the deficiency $\delta(X;p)$
in \eqref{eqn:deficiency-unit-def}
to the value $L(p+ \chi_X) - L(p)$.

\begin{prop}[{\cite[Lemma 5.8]{MSY16}}]
\label{prop:Lyap-diff}
 For $p \in \Z^n_+$ and $X\subseteq N$, it holds that 
 \begin{equation}
 \label{eqn:Lyap-diff-0} 
  L(p+ \chi_X) - L(p)  = |X| - |O(X,p)|
\ \  (= -\delta(X;p) ).
 \end{equation}
\end{prop}

The relation \eqref{eqn:Lyap-diff-0} enables us to reformulate 
the notions of overdemanded and excess-demand sets 
in terms of the Lyapunov function $L$.
It  follows immediately from \eqref{eqn:Lyap-diff-0} that
an overdemanded set $X$
corresponds to a descent direction $d= \chi_X$ of the Lyapunov function,
as follows.

\begin{prop}
\label{prop:overdemand-Lyap}
Let $p \in \Z^n_+$.
A set $X\subseteq N$  is an overdemanded set at price $p$
in the sense of \eqref{eqn:overdemand-unit-def}
if and only if $L(p+\chi_X) < L(p)$ holds.
In particular,  $X$  is  a minimal overdemanded set at price $p$
if and only if $X$ is a minimal subset of $N$ satisfying $L(p+\chi_X) < L(p)$.
\end{prop}

 Based on this relationship, the (generic)
algorithm {\sc Ascending Auction}
can be rewritten in terms of the Lyapunov function as follows.

\medskip

\noindent
{\bf Algorithm} {\sc Lyapunov Minimization} (generic form)
\\
Step~0:  Set $p:=p^\circ$, where $p^\circ \in \Z^n_+$
is a vector satisfying
$p^\circ \le \pminmin$. 
\\
Step~1:
 If $L(p+\chi_X) \ge L(p)$ for all $X\subseteq N$, then stop.
\\ 
Step~2:
 Find a set $X\subseteq N$ such that $L(p+\chi_X) < L(p)$.
\\
Step~3:
 Set $p:=p + \chi_X$ and go to Step~1.

\medskip

We are concerned with how to select the item set $X$ in Step~2
that ensures the minimal equilibrium price vector $\pminmin$
at the termination of the algorithm.
More specifically, we aim to formulate propositions
for {\sc Lyapunov Minimization}
that correspond to 
those for {\sc Ascending Auction}
stated in Section~\ref{sec:review-auction}, namely,
Theorem~\ref{thm:min-overdemand-unit} for the choice of minimal overdemanded sets,
Theorem~\ref{thm:max-excess} for maximal excess-demand sets,
and Theorem~\ref{thm:excess} for (general) excess-demand sets.

With the characterization of a minimal overdemanded set
given in Proposition~\ref{prop:overdemand-Lyap},
the following proposition for {\sc Lyapunov Minimization} is immediate from 
Theorem~\ref{thm:min-overdemand-unit} for {\sc Ascending Auction}.

\begin{prop}
\label{prop:overdemand-Lyap-algo}
The algorithm {\sc Lyapunov Minimization} finds 
the minimal minimizer  $\pminmin$ of the Lyapunov function $L$
in \eqref{eqn:Lyap-unit}
if, in each iteration, the item set $X$ selected in Step~2
is a minimal subset of $N$ satisfying $L(p+\chi_X) < L(p)$.
\end{prop}

The (unique) maximal excess-demand set can be characterized in terms of the Lyapunov function
as its (unique) minimal steepest descent direction.

\begin{prop}[{\cite[Lemma 5.6(i)]{MSY16}}]
\label{prop:uniq-max-excess-minimizer}
Let $p \in \Z^n_+$.
A set $X\subseteq N$  is the (unique) maximal excess-demand set 
\eqref{eqn:def-excess-demand} at price $p$
if and only if $X$ is the (unique) minimal minimizer
of $L(p+\chi_X) -L(p)$.
\end{prop}

With this characterization of the maximal  excess-demand set,
the following proposition for {\sc Lyapunov Minimization} is immediate from 
Theorem~\ref{thm:max-excess} for {\sc Ascending Auction}.

\begin{prop}
\label{prop:max-excess-Lyap}
The algorithm {\sc Lyapunov Minimization} finds 
the minimal minimizer  $\pminmin$ of 
the Lyapunov function $L$
in \eqref{eqn:Lyap-unit}
if, in each iteration, the item set $X$ selected in Step~2
is the minimal minimizer of $L(p+\chi_X) -L(p)$.
\end{prop}

\medskip

To characterize excess-demand sets, we consider the condition 
\begin{equation}
\label{eqn:set-cond3-Lyap} 
L(p + \chi_Y) > L(p+ \chi_X) \qquad (\forall Y \subsetneq X)
\end{equation}
for $X \subseteq N$.
We say that a nonempty set $X$ is \textit{$(L,p)$-minimal}
if it satisfies \eqref{eqn:set-cond3-Lyap}.
First, we observe the following facts.
\begin{itemize}
\item
\eqref{eqn:set-cond3-Lyap} for $Y = \emptyset$ implies $L(p ) > L(p+ \chi_X)$.

\item
A minimal excess-demand set at $p$ is $(L,p)$-minimal.
\\
(Proof)
By Proposition~\ref{prop:excess-demand-overdemanded}(3),
a minimal excess-demand set is a minimal overdemanded set,
which is,
by Proposition~\ref{prop:overdemand-Lyap},
a minimal subset $X$ satisfying 
$L(p+\chi_X) < L(p)$.
The minimality of $X$ implies that there exists no proper subset $Y$ of $X$
satisfying 
$L(p + \chi_Y) \le L(p+ \chi_X)$.
Thus, $X$ satisfies \eqref{eqn:set-cond3-Lyap}.

\item
The maximal excess-demand set at $p$ is $(L,p)$-minimal. 
\\
(Proof)
By Proposition~\ref{prop:uniq-max-excess-minimizer},
the maximal excess-demand set $X$
is the minimal minimizer of $L$,
which implies that there exists no proper subset $Y$ of $X$
satisfying 
$L(p + \chi_Y) \le L(p+ \chi_X)$.
Thus, $X$ satisfies \eqref{eqn:set-cond3-Lyap}.

\end{itemize}

Next, we show that every  excess-demand set is
$(L,p)$-minimal, and vice versa.

\begin{prop}
\label{prop:excess-demand-cond}
Let $p \in \Z^n_+$.
A set $X\subseteq N$  is an excess-demand set 
\eqref{eqn:def-excess-demand}
at price $p$
if and only if $X$ is $(L,p)$-minimal
in \eqref{eqn:set-cond3-Lyap}.
\end{prop}
\begin{proof}
By the definition \eqref{eqn:def-excess-demand},
$X$ is an excess-demand set at $p$ if (and only if)
the inequality
$|U(Z, p) \cap O(X, p)| - |Z| > 0$
holds for all nonempty $Z \subseteq X$.
Since
\begin{align*}
& |U(Z, p) \cap O(X, p)| - |Z| \\
& = |O(X, p)| - |O(X \setminus Z, p)| - |X| + |X \setminus Z|
\qquad \mbox{(by \eqref{eqn:def-excess-demand3a}, \eqref{eqn:def-excess-demand3b})}
\\
& = (|X \setminus Z|- |O(X \setminus Z, p)|) - (|X|- |O(X, p)|)
\\
& = L(p+\chi_{X \setminus Z}) - L(p+\chi_X)
\qquad \mbox{(by \eqref{eqn:Lyap-diff-0})},
\end{align*}
$X$ is an excess-demand set
if and only if
$L(p + \chi_{X\setminus Z}) > L(p+ \chi_X)$
for all nonempty 
$Z \subseteq X$.
The latter condition is equivalent to 
\eqref{eqn:set-cond3-Lyap}.
\end{proof}

With the characterization of an excess-demand set
given in Proposition~\ref{prop:excess-demand-cond}, 
the following proposition for {\sc Lyapunov Minimization} is immediate from 
Theorem~\ref{thm:excess} for {\sc Ascending Auction}.

\begin{prop}
\label{prop:excess-Lyap}
The algorithm {\sc Lyapunov Minimization} finds 
the minimal minimizer  $\pminmin$ of 
the Lyapunov function $L$
in \eqref{eqn:Lyap-unit}
if, in each iteration, the item set $X$ selected in Step~2 
is $(L,p)$-minimal
in \eqref{eqn:set-cond3-Lyap}.
\end{prop}

It may be worth noting that $(L,p)$-minimality
is, in fact, the minimality 
with respect to a partial order $\prec$ on $2^N$
defined by:
$Y \prec Z$
$\Leftrightarrow$
[$Y \subseteq Z$ and $L(p+ \chi_{Y}) \le L(p+ \chi_Z)$].

 The correspondence of the item sets used in
the algorithms {\sc Ascending Auction} and 
{\sc Lyapunov Minimization} is summarized in
Table \ref{table:unit-demand}.

\begin{table} 
\begin{center}
\caption{Characterizations of the item sets for the unit-demand model}
\label{table:unit-demand}

\medskip

\addtolength{\tabcolsep}{-2pt}
 \begin{tabular}{lcl}
 {\sc Ascending Auction} &  & {\sc Lyapunov Minimization}
\\ 
\hline
minimal overdemanded set & $\Longleftrightarrow$ & minimal $X$ with $L(p+\chi_X) < L(p)$ 
\\ 
 \qquad $\Updownarrow$ Prop.~\ref{prop:excess-demand-overdemanded}(3) & 
Prop.~\ref{prop:overdemand-Lyap} &
\\ 
minimal excess-demand set &   &  
\\  \hline
maximal excess-demand set  & $\Longleftrightarrow$ & minimal minimizer of $L(p+\chi_X) - L(p)$ 
\\ 
 \qquad $\Updownarrow$ Prop.~\ref{prop:excess-demand-overdemanded}(5) & 
Prop.~\ref{prop:uniq-max-excess-minimizer} &
\\ 
\multicolumn{2}{l}{minimal maximizer of deficiency}   &  
\\ \hline
 excess-demand set & $\Longleftrightarrow$ 
&    $(L,p)$-minimal set, i.e., $X$  satisfying 
\\ 
 & Prop.~\ref{prop:excess-demand-cond} 
&   \, \,   $L(p+\chi_Y) > L(p+ \chi_X)$ \   $(\forall Y \subsetneq X)$  
\\ \hline
 \end{tabular}
\addtolength{\tabcolsep}{+2pt}
\end{center}
\end{table}

\section{Minimization of general \Lnat-convex functions}
\label{sec:Lnat-minimize}


In Section~\ref{sec:Lyap-unit}
we introduced the Lyapunov function for the unit-demand model
and provided conditions guaranteeing
the algorithm of {\sc Lyapunov Minimization}
to find the minimal minimizer 
(Propositions \ref{prop:overdemand-Lyap-algo}, \ref{prop:max-excess-Lyap}, 
and \ref{prop:excess-Lyap}).
Motivated by the fact that
the Lyapunov function 
is an \Lnat-convex function,
we extend, in this section, the algorithm of {\sc Lyapunov Minimization}
to general \Lnat-convex functions.
This generalization enables us, in Section~\ref{sec:Lyap-multi}, 
to extend the ascending auction algorithm
for the unit-demand model (Section~\ref{sec:review-auction}) 
to the algorithm for the multi-demand model
described in Section~\ref{sec:auction-multi}.

 A function $g: \Z^n \to \Rinf$ defined on integer lattice points 
is said to be \textit{\Lnat-convex} 
if for every $p, q \in \dom g$ and every nonnegative $\lambda \in \Z_+$,
it holds that
\begin{equation}
\label{eqn:def-Lnat}
  g(p) + g(q) \geq g((p+ \lambda \1) \wedge q) + g(p \vee (q - \lambda \1)),
\end{equation}
where $\dom g= \{p \in \Z^n \mid g(p) < + \infty\}$,
$\1 = (1,1,\ldots, 1)$,
and for any $p, q \in \Z^n$,  $p \wedge q$
and $p \vee q$ denote, respectively, the vectors
of component-wise minimum and maximum of $p$ and $q$.
It follows from \eqref{eqn:def-Lnat} with $\lambda =0$ that
an \Lnat-convex function is a submodular function, 
that is,
\begin{equation}
\label{eqn:subm-Lnat}
g(p) + g(q) \geq g(p \wedge q) + g(p \vee q)
\end{equation}
for every $p,q \in \dom g$.
 The concept of \Lnat-convex function plays a primary role
in the theory of discrete convex analysis \cite{Murota03book,Murota24}.
 The Lyapunov function $L$ 
associated 
with the unit-demand auction model by \eqref{eqn:Lyap-unit}
is known to be \Lnat-convex (\cite{MSY13,MSY16}; see also \cite{Murota16,Murota24}),
where the function $L$ is integer-valued and $\dom L = \Z^n_+$
in our model.

Throughout this section, we assume that $g: \Z^n \to \Rinf$ is an \Lnat-convex function 
such that the set of its minimizers, denoted as $\arg\min g$,  
is nonempty and bounded from below.
 Under these conditions, $\arg\min g$ has a unique minimal vector,
 which is a consequence of the submodularity \eqref{eqn:subm-Lnat}.
We denote the unique minimal minimizer of $g$ by $\pminminL$.
 Note that the Lyapunov function $L$ in \eqref{eqn:Lyap-unit} satisfies the above-mentioned conditions
and that the (unique) minimal minimizer of $L$ coincides with the (unique) equilibrium price vector $\pminmin$.

The following (generic) algorithm is
a straightforward (or formal) generalization
of {\sc Lyapunov Minimization}
in Section~\ref{sec:Lyap-unit}
to an arbitrary \Lnat-convex function $g$.

\medskip

\noindent
{\bf Algorithm}  {\sc \Lnat-convex Minimization}  (generic form)
\\
Step~0:  Set $p:=p^\circ$, where
$p^\circ \in \dom g$ is a vector satisfying
$p^\circ \le \pminminL$.
\\
Step~1:
If $g(p+ \chi_X) \ge g(p)$ for all $X \subseteq N$, then stop.
\\ 
Step~2:
 Find a set $X \subseteq N$ such that $g(p+ \chi_X) < g(p)$.
\\
Step~3:
 Set $p:=p + \chi_X$ and go to Step~1.

\medskip

As in {\sc Lyapunov minimization}, 
the selection of the set $X$ in Step~2 is crucial 
in obtaining the minimal minimizer $\pminminL$.
In analogy to \eqref{eqn:set-cond3-Lyap} 
characterizing excess-demand sets,
we consider the condition 
\begin{equation}
\label{eqn:set-cond3-Lnat} 
g(p + \chi_Y) > g(p+ \chi_X) \qquad (\forall Y \subsetneq X)
\end{equation}
for $X \subseteq N$
and refer to a nonempty set $X$ 
satisfying this condition as a \textit{$(g,p)$-minimal set}.
This condition, when imposed on $X$ in Step~2, guarantees
that $\pminminL$ is found at the termination of the algorithm,
which is stated in Theorem~\ref{thm:excess-Lnat} below.
Note that this theorem corresponds to 
Propositions
\ref{prop:overdemand-Lyap-algo}, \ref{prop:max-excess-Lyap}, and
\ref{prop:excess-Lyap} 
for
{\sc Lyapunov minimization}
for the unit-demand model.

\begin{theorem}
\label{thm:excess-Lnat}
The algorithm {\sc \Lnat-convex Minimization}
finds the minimal minimizer  $\pminminL$ of an \Lnat-convex function 
$g: \Z^n \to \Rinf$
if, in each iteration, the set $X$ selected in Step~2
is $(g,p)$-minimal
in the sense of \eqref{eqn:set-cond3-Lnat},
which is the case if
\\
{\rm (i)}
$X$ is a minimal subset of $N$ satisfying $g(p+\chi_X) < g(p)$, or 
\\
{\rm (ii)}
$X$ is the unique minimal minimizer of $g(p+\chi_X) - g(p)$.
\end{theorem}

\begin{proof}
By Lemma \ref{lem:donotexceed} below,
the inequality $p \le \pminminL$ is maintained in each iteration.
Then, it follows from \eqref{eqn:def-Lnat}
and Lemma~\ref{lem:localopt} below
that the termination criterion in Step~1 is satisfied only when $p = \pminminL$.
Therefore, the output of the algorithm coincides with $\pminminL$.
The $(g,p)$-minimality of $X$ in cases (i) and (ii)
is immediate from the definition \eqref{eqn:set-cond3-Lnat}. 
\end{proof}

\begin{lemma}
\label{lem:donotexceed}
 Let $p \in \dom g$  be a vector such that $p \le \pminminL$ and $p \ne \pminminL$.
 For every $(g,p)$-minimal set $X \subseteq N$,
it holds that $p + \chi_X\le \pminminL$.
\end{lemma}

\begin{proof}
To prove by contradiction,
we assume that the inequality $p + \chi_X\le \pminminL$ fails
for some $(g,p)$-minimal set $X$.
 Let
\[
 q = p + \chi_X, \qquad
Z = \{ i \in N \mid q(i) > \pminminL(i)\}.
\]
By our assumption, we have $Z \ne \emptyset$.  
Since
\begin{equation}
\label{eqn:Zmax}
Z = \{ i \in N \mid q(i) = \pminminL(i) + 1 \}
= \arg\max_{i \in N} \{q(i) - \pminminL(i)\},
\end{equation}
\Lnat-convexity of $g$ implies that 
\begin{align}
\label{eqn:Lnat-exc-2}
 g(q) + g(\pminminL) \ge  g(q - \chi_Z) + g(\pminminL + \chi_Z)
\end{align}
(see \cite[Theorem 7.7]{Murota03book}).
 Since $\pminminL$ is a minimizer of $g$, we have 
$g(\pminminL) \le  g(\pminminL + \chi_Z)$, which, together with \eqref{eqn:Lnat-exc-2},
implies that $g(q)  \ge  g(q - \chi_Z)$.
Since $Z \subseteq X$ by \eqref{eqn:Zmax} and $p \le \pminminL$,
this inequality $g(q)  \ge  g(q - \chi_Z)$ can be rewritten as
$g(p + \chi_X) \ge g(p + \chi_{X \setminus Z})$,
a contradiction to the condition \eqref{eqn:set-cond3-Lnat} with $Y =X \setminus Z$.
\end{proof}

The following fact is implicit in 
Murota and Shioura \cite[Proof of Theorem~1.3]{MS14exbndLmin}
and explicit in Shioura \cite[Theorem~2.7]{Shi17L} (without proof).
For completeness we provide a self-contained proof.

\begin{lemma}
\label{lem:localopt}
Let $g: \Z^n \to \Rinf$
be an \Lnat-convex function and 
$\hat p \in \dom g$  be a vector such that 
there exists a minimizer $\pminL$ of $g$ satisfying $\hat p \le \pminL$.
Then $\hat p$ is a minimizer of $g$ if and only if
$g(\hat p) \leq g(\hat p + \chi_{X})$ for all $X \subseteq N$.
\end{lemma}

\begin{proof}
The local optimality criterion for \Lnat-convex functions
(see \cite[Theorem 7.14]{Murota03book}) states that
$\hat p$ is a minimizer of $g$ if and only if
$g(\hat p) \leq g(\hat p + \chi_{X})$ and
$g(\hat p) \leq g(\hat p - \chi_{X})$ for all $X \subseteq N$.
The former condition is satisfied by the assumption and the latter can be shown as follows.
Consider the inequality \eqref{eqn:def-Lnat}
for $p = \hat p - \chi_{X}$, $q = \pminL$, and $\lambda =1$:
\[
 g(\hat p - \chi_{X}) + g(\pminL) \geq g((p+  \1) \wedge q) + g(p \vee (q -  \1)).
\]
Here, 
$(p+  \1) \wedge q = (\hat p - \chi_{X} + \1) \wedge \pminL =  \hat p + \chi_{Z}$ 
for some $Z  \subseteq N$ and hence
\[
 g((p+  \1) \wedge q) 
= g(\hat p + \chi_{Z}) \geq  g(\hat p),
\]
whereas 
$ g(p \vee (q -  \1)) \ge g(\pminL)$ 
since $\pminL$ minimizes $g$.
Therefore,
\[
 g(\hat p - \chi_{X}) + g(\pminL) \geq g(\hat p) + g(\pminL),
\]
which shows
$ g(\hat p - \chi_{X}) \geq g(\hat p)$.
\end{proof}

\begin{remark} \rm
\label{rem:gpminfam-unimax}
We mention the following facts, 
although we do not need them in our argument:
\par
(1) The family of $(g,p)$-minimal sets has a unique maximal set.
\par
(2) 
 The maximal $(g,p)$-minimal set coincides with the minimal minimizer of $g(p + \chi_{X})$.
\\
These statements correspond, respectively, to
Proposition~\ref{prop:excess-demand-overdemanded}(4)
and 
Proposition~\ref{prop:uniq-max-excess-minimizer}
on the Lyapunov function $L$
for the unit-demand model in \eqref{eqn:Lyap-unit},
where 
Proposition~\ref{prop:excess-demand-cond}
allows us to replace ``excess-demand set'' in those propositions
with ``$(L,p)$-minimal set.''
For completeness, the proofs of (1) and (2) are given in Appendix. 
\qed
\end{remark}

\begin{remark} \rm
\label{rem:runningtimeL}
The running time of the 
{\sc \Lnat-convex Minimization}
algorithm 
using a $(g,p)$-minimal set $X$ in Step~2
is discussed here.
The number of iterations (updates) lies between 
$\| p^\circ - \pminminL \|_{\infty}$
and
$\| p^\circ - \pminminL \|_{1}$ $(\le n \cdot \| p^\circ - \pminminL \|_{\infty})$.
If, in each iteration,
$X$ is 
the maximal $(g,p)$-minimal set 
(i.e., the unique minimal minimizer of $g(p+\chi_X)$; see Remark~\ref{rem:gpminfam-unimax}),
the number of iterations 
equals the lower bound
$\| p^\circ - \pminminL \|_{\infty}$;
see \cite{MSY16}.
In each iteration, we have to find a $(g,p)$-minimal set.
We can compute the maximal $(g,p)$-minimal set in polynomial time
via submodular function minimization.
We can also compute a minimal $(g,p)$-minimal set in polynomial time
with $O(n\sp{2})$ 
applications of submodular function minimization.
It is not known whether we can compute faster
some other $(g,p)$-minimal set
(without imposing maximality or minimality).
\qed
\end{remark}

\section{Lyapunov function minimization for multi-demand model}
\label{sec:Lyap-multi}


The objective of this section is to complete the proofs 
of Theorems \ref{thm:min-overdemand-multi} and \ref{thm:excess-multi}
on the ascending auction algorithm for the multi-demand model.
Our approach is to extend
the dual algorithmic view,
namely, the equivalence of 
 {\sc Ascending Auction}  
and
 {\sc Lyapunov Minimization} 
algorithms discussed in Section~\ref{sec:Lyap-unit}
for the unit-demand model
to that for the multi-demand model.
Technically, we rely on the general results on 
{\sc \Lnat-convex Minimization} algorithm
obtained in Section~\ref{sec:Lnat-minimize}.

\subsection{{\sc Lyapunov Minimization} for multi-demand model}
\label{sec:Lyap-alg-multi}


For the multi-demand model,
the \textit{Lyapunov function}  
is defined as 
\begin{equation}
\Lm(p) =
\sum_{b \in M} \max_{x \in [\0, u]_\Z}\{v_b(x) - p^\top x \} + p^\top u
\qquad (p \in \Z^n_+),
\label{eqn:def-Lyap-multi}
\end{equation}
where $p^\top x = \sum_{i =1}\sp{n} p(i) x(i)$
and $p^\top u = \sum_{i=1}^n  p(i)u(i)$.
When the bidders' valuation functions
enjoy \Mnat-concavity (strong gross substitutability),
this Lyapunov function is \Lnat-convex and captures the equilibrium prices as its minimizers,
as follows.

\begin{prop}[\cite{Aus06,MSY13,MSY16,SY09}]
\label{prop:multi-Lyap-prop}
 Let $L: \Z^n_+ \to \Z$ 
be the Lyapunov function in \eqref{eqn:def-Lyap-multi} for the multi-demand model 
and assume that the valuation functions $v_b$ are \Mnat-concave for all $b \in M$.
\\
{\rm (1)}
The Lyapunov function $L$  is \Lnat-convex.
\\
{\rm (2)}
$p \in \Z^n_+$ is an  equilibrium price vector if and only if 
it is a  minimizer of $L$.
In particular,
$p \in \Z^n_+$ is the unique minimal equilibrium price vector
if and only if 
it is the unique minimal minimizer of $L$.
\end{prop}

Given the general results 
(Section~\ref{sec:Lnat-minimize})
on \Lnat-convex function minimization,
Proposition~\ref{prop:multi-Lyap-prop} naturally suggests that we
apply the {\sc Lyapunov Minimization} algorithm 
to the function $L$ 
in \eqref{eqn:def-Lyap-multi}.
In this case, the key condition
\eqref{eqn:set-cond3-Lnat} 
for the choice of $X$ in Step~2 of the {\sc \Lnat-convex Minimization} algorithm
reads
\begin{equation}
\label{eqn:set-cond3-Lyap-multi} 
\Lm(p + \chi_Y) > \Lm(p+ \chi_X) \qquad (\forall Y \subsetneq X).
\end{equation}
Recall that a nonempty set $X$ satisfying this condition is called $(L,p)$-minimal.

Under the assumption of \Mnat-concavity of valuation functions,
the Lyapunov function $L$ is \Lnat-convex
by Proposition~\ref{prop:multi-Lyap-prop}(1).
Then we can use  Theorem~\ref{thm:excess-Lnat}
for general \Lnat-convex minimization
to obtain the following theorem,
stating that the output of 
the algorithm {\sc Lyapunov Minimization} 
is indeed the unique minimal minimizer $\pminmin$  of $L$.

\begin{theorem}
\label{thm:excess-Lyap-multi}
The algorithm {\sc Lyapunov Minimization}
applied to the Lyapunov function $\Lm$
in \eqref{eqn:def-Lyap-multi}
for the multi-demand model with \Mnat-concave valuations
 finds the unique minimal minimizer  $\pminmin$ of $\Lm$
if, in each iteration, the item set $X$ selected in Step~2
is $(L,p)$-minimal
in the sense of \eqref{eqn:set-cond3-Lyap-multi},
which is the case if
\\
{\rm (i)}
$X$ is a minimal subset of $N$ satisfying 
$\Lm(p+\chi_X) < \Lm(p)$, or 
\\
{\rm (ii)}
$X$ is the unique minimal minimizer of $\Lm(p+\chi_X) - \Lm(p)$.
\end{theorem}

\subsection{{\sc Ascending Auction} for multi-demand model}
\label{sec:ascend-alg-multi}


In this section we translate the {\sc Lyapunov Minimization} algorithm
in Section~\ref{sec:Lyap-alg-multi}
to the form of {\sc Ascending Auction}
by characterizing the deficiency, overdemanded sets, and excess-demand sets
in terms of the Lyapunov function $L$.

First we consider overdemanded sets
and deficiency, which are defined, respectively, by
\begin{align*}
& \sum_{b \in M} \mu_b(X;p) > u(X),
& \cdots \ \eqref{eqn:overdemand-multi-def}
\\
& \delta(X;p)=
\sum_{b \in M} \mu_b(X;p) - u(X).
& \cdots \  \eqref{eqn:deficiency-multi-def}
\end{align*}

The next proposition connects
the deficiency $\delta(X;p)$
to the value of $L(p+ \chi_X) - L(p)$.

\begin{prop}[{\cite[Lemma 1.1]{MSY13}; see also \cite{Aus06}}]
\label{prop:mLyap-diff}
For $p \in \Z^n_+$ and $X\subseteq N$, it holds that
\begin{equation}
\label{eqn:Lyap-diff-multi} 
  L(p+ \chi_X) - L(p)
 = u(X) - \sum_{b \in M} \mu_b(X;p)
\ \  (= -\delta(X;p) ).
\end{equation}
\end{prop}

The following proposition is an immediate consequence of this relation.

\begin{prop}
\label{prop:mLyap-diff-cor}
Let $p \in \Z^n_+$ and $X\subseteq N$. 
\\
{\rm (1)}
$L(p+\chi_X) < L(p)$ holds 
if and only if
$X$ is overdemanded 
in the sense of \eqref{eqn:overdemand-multi-def}.
\\
{\rm (2)}
$X$ is a minimizer of $L(p+\chi_X) - L(p)$ 
if and only if
$X$ is a maximizer of deficiency 
$\delta(X;p)$ in \eqref{eqn:deficiency-multi-def}.
In particular, 
the (unique) minimal minimizer of $L(p+\chi_X) - L(p)$ 
is the (unique) minimal maximizer of the deficiency.
\end{prop}

We next connect excess-demand sets to the Lyapunov function.
Recall that 
a nonempty item set $X$ is called an excess-demand set
if the following condition is satisfied:
\begin{align*}
&  \sum_{b \in M} (\mu_b(X;p) - \mu_b(X\setminus Z;p)) > u(Z) 
 \qquad (\emptyset \ne  \forall Z \subseteq X).
& \cdots \ 
\eqref{eqn:excess-demand-multi-def}
\end{align*}

\begin{prop}
\label{prop:excess-demand-cond-multi} 
Let $p \in \Z^n_+$.
A set $X\subseteq N$  is an excess-demand set at price $p$
in the sense of \eqref{eqn:excess-demand-multi-def}
if and only if $X$ is $(L,p)$-minimal 
in the sense of \eqref{eqn:set-cond3-Lyap-multi}.
\end{prop}
\begin{proof}
Obviously, 
\eqref{eqn:set-cond3-Lyap-multi}
is equivalent to 
\begin{equation}
\label{eqn:set-cond3-Lyap-multi2} 
L(p + \chi_{X\setminus Z}) > L(p+ \chi_X) \qquad (\emptyset \ne  Z \subseteq X).
\end{equation}
By \eqref{eqn:Lyap-diff-multi} we have
\begin{align*}
& L(p + \chi_{X\setminus Z}) - L(p+ \chi_X)
\\
&
= [u(X\setminus Z) - \sum_{b \in M} \mu_b(X\setminus Z;p)]
-
[u(X) - \sum_{b \in M} \mu_b(X;p)]
\\
&
= - u(Z) 
+ \sum_{b \in M} (\mu_b(X;p) - \mu_b(X\setminus Z;p)).
\end{align*}
 Hence, \eqref{eqn:set-cond3-Lyap-multi2} holds
if and only if 
the condition \eqref{eqn:excess-demand-multi-def}
for excess-demand sets is satisfied.
\end{proof}

Using Propositions \ref{prop:mLyap-diff-cor} and \ref{prop:excess-demand-cond-multi},
we can translate Theorem \ref{thm:excess-Lyap-multi} to the following theorem
concerning the {\sc Ascending Auction} algorithm.

\begin{theorem}
\label{thm:min-overdemand-excess-multi}
 The algorithm {\sc Ascending Auction} 
applied to the multi-demand model 
with \Mnat-concave valuations
finds the unique minimal equilibrium price vector  $\pminmin$
if, in each iteration, the item set $X$ selected in Step~2
is an excess-demand set
in the sense of \eqref{eqn:excess-demand-multi-def},
which is the case if
\\
{\rm (i)}
$X$ is a minimal overdemanded set 
\eqref{eqn:overdemand-multi-def}
at price~$p$, or 
\\
{\rm (ii)}
$X$ is the unique minimal maximizer of the deficiency $\delta(X;p)$
in \eqref{eqn:deficiency-multi-def}. 
\end{theorem}

The special cases (i) and (ii) of
Theorem \ref{thm:min-overdemand-excess-multi} above
correspond to
Theorem \ref{thm:min-overdemand-multi},
while the general case of
Theorem \ref{thm:min-overdemand-excess-multi}
is nothing but Theorem \ref{thm:excess-multi}. 
We have thus completed the proofs of the two main theorems announced in Section~\ref{sec:auction-multi}.

 The correspondence of the item sets used in
the {\sc Ascending Auction} algorithm
for the multi-demand model and the
{\sc Lyapunov Minimization} algorithm
applied to $L$ in \eqref{eqn:def-Lyap-multi}
 is summarized in
Table \ref{table:multi-demand}.

\begin{table} 
\begin{center}
\caption{
Item sets used in the ascending auction algorithm for the multi-demand model}
\label{table:multi-demand}

\medskip

\addtolength{\tabcolsep}{-2pt}
 \begin{tabular}{lcl}
 {\sc Ascending Auction}  &  & {\sc Lyapunov Minimization} 
\\
 \quad for multi-demand model &  & \quad applied to $L$ in \eqref{eqn:def-Lyap-multi}
\\ 
\hline
minimal overdemanded set \eqref{eqn:overdemand-multi-def} 
& $\Longleftrightarrow$ 
& minimal $X$ with 
\\ 
& Prop.~\ref{prop:mLyap-diff-cor}(1) 
& \quad $L(p+\chi_X) < L(p)$ 
\\  \hline
minimal maximizer of 
 & $\Longleftrightarrow$ 
 & minimal minimizer of 
\\ 
\quad  deficiency $\delta(X;p)$ in \eqref{eqn:deficiency-multi-def} & 
Prop.~\ref{prop:mLyap-diff-cor}(2) 
& \quad $L(p+\chi_X) - L(p)$ 
\\ \hline
 excess-demand set  \eqref{eqn:excess-demand-multi-def} & $\Longleftrightarrow$ 
&     $(L,p)$-minimal set, i.e., $X$  satisfying  
\\ 
 & Prop.~\ref{prop:excess-demand-cond-multi} 
&    \, \,   $L(p+\chi_Y) > L(p+ \chi_X)$ \   $(\forall Y \subsetneq X)$ 
\\ \hline
 \end{tabular}
\addtolength{\tabcolsep}{+2pt}
\end{center}
\end{table}

\appendix
\section{Appendix: Technical supplement to Remark \ref{rem:gpminfam-unimax}}
\label{sec:appendix}

In this appendix, we complete 
Remark~\ref{rem:gpminfam-unimax}
by proving the following proposition.

\begin{prop}
\label{prop:excess-Lnat-Rem53}
Let $\calF_p$ 
be the family of $(g,p)$-minimal sets
for an \Lnat-convex function $g: \Z^n \to \Rinf$.

\noindent
{\rm (1)} 
$\calF_p$ has a unique maximal set.

\noindent
{\rm (2)} 
The maximal member of $\calF_p$ coincides with the minimal minimizer of $g(p + \chi_{X})$.
\end{prop}

The following proposition serves as a technical basis for the proof of this proposition.

\begin{prop}
\label{prop:excess-Lnat-union}
Let $\calF_p$ 
be the family of $(g,p)$-minimal sets
for an \Lnat-convex function $g: \Z^n \to \Rinf$.

\noindent
{\rm (1)} 
If $X \in \calF_p$ and $Z \in \calF_p$,
then $X \cup Z \in \calF_p$.

\noindent
{\rm (2)} 
If $X \in \calF_p$, $Z \subseteq N$, and $X\setminus Z \ne \emptyset$,
 then $ g(p+ \chi_{Z}) > g(p+ \chi_{X \cup Z})$. 
\end{prop}

\begin{proof}
Let $\rho(Y)=g(p+\chi_Y)$ for $Y \subseteq N$. 

(1)
To prove $X \cup Z \in \calF_p$, 
we show the following inequality:
\begin{align}
& \rho(Y) > \rho(X \cup Z)  \qquad (\forall Y \subsetneq X  \cup Z).
\label{eqn:excess-Lnat-union:6}
\end{align}
Since 
$Y \subseteq X  \cup Z$ and $Y \ne X \cup Z$,
at least one of 
$Y \cap  X \subsetneq X$ and $Y \cap  Z \subsetneq Z$ holds.
By symmetry, we assume the former.
By submodularity of $\rho$, it holds that
\begin{align*}
& \rho(Y) + \rho(X)  + \rho(Z)
\notag\\
& \ge \rho(Y \cap X) + \rho(Y \cup X)  + \rho(Z)
\notag\\
& \ge \rho(Y \cap X) + \rho((Y \cup X)\cap Z)  + \rho(X \cup Z),
\end{align*}
from which it follows that
\begin{align}
& \rho(Y)  - \rho(X \cup Z)
\notag\\
& \ge [\rho(Y \cap X)  -  \rho(X) ]
+ [\rho((Y \cup X)\cap Z)  -  \rho(Z) ].
 \label{eqn:excess-Lnat-union:7}
\end{align}
Here we have 
$\rho(Y \cap X ) - \rho(X) >0$ 
since $Y \cap X \subsetneq X$ and $X \in \calF_p$.
 Similarly, 
we have
$\rho((Y \cup X)\cap Z)  -  \rho(Z)  \ge 0$ 
since $(Y \cup X)\cap Z \subseteq Z$ and $Z \in \calF_p$.
 Hence, the right-hand side of \eqref{eqn:excess-Lnat-union:7} is positive,
implying the desired inequality $\rho(Y)  > \rho(X \cup Z)$
in \eqref{eqn:excess-Lnat-union:6}.

(2)
 Since $X \in \calF_p$, we have
\begin{align}
& \rho(Y) > \rho(X)  \qquad (\forall Y \subsetneq X) .
\label{eqn:excess-Lnat-union:1}
\end{align}
\Lnat-convexity of $g$ implies that $\rho$ is submodular,
and therefore 
\begin{align}
 \label{eqn:excess-Lnat-union:3}
& \rho(X) + \rho(Z) \ge \rho(X \cap Z) + \rho(X \cup Z)
\end{align}
holds.
 Since $X \cap Z \subsetneq X$ by $X\setminus Z \ne \emptyset$,
 we have 
$\rho(X) < \rho(X \cap Z)$ by \eqref{eqn:excess-Lnat-union:1},
which, together with \eqref{eqn:excess-Lnat-union:3}, implies 
$\rho(Z) > \rho(X \cup Z)$, 
i.e.,
$ g(p+ \chi_{Z}) > g(p+ \chi_{X \cup Z})$.
\end{proof}

We are now ready to prove Proposition~\ref{prop:excess-Lnat-Rem53}.
The first claim (1) is immediate from Proposition~\ref{prop:excess-Lnat-union}(1).
The proof of the second claim (2) is as follows.
Let $Z$ be the minimal minimizer of $g(p + \chi_{X})$.
Then there exists no proper subset $Y$ of $Z$ satisfying 
$g(p + \chi_Y) \le g(p+ \chi_Z)$,
that is, $Z \in \calF_p$.
Let $W$ be the maximal member of $\calF_p$. 
Since $Z \in \calF_p$ 
and $W$ is maximal in $\calF_p$,
we have
$W \supseteq Z$.
If $W \setminus Z \ne \emptyset$,
we must have
$g(p+ \chi_{Z}) > g(p+ \chi_{W \cup Z})$
by Proposition~\ref{prop:excess-Lnat-union}(2).
This contradicts the definition of
$Z$ that $Z$ is a minimizer of $g(p + \chi_{X})$.
Hence $W = Z$.
This completes the proof of Proposition~\ref{prop:excess-Lnat-Rem53}.

\noindent {\bf Acknowledgement}.
This work was supported by JSPS/MEXT KAKENHI JP23K11001 and JP23K10995,
and by JST ERATO Grant Number JPMJER2301, Japan.

\end{document}